\RequirePackage{fix-cm}
\documentclass[smallextended]{svjour3} 
\smartqed
\usepackage{graphicx}
\usepackage{epsfig,color,xcolor,subfigure,amsmath,amssymb,multirow}
\usepackage[centering,a4paper]{geometry}

\colorlet{mylinkcolor}{green!50!black}
\colorlet{mycitecolor}{red!50!black}
\colorlet{myviewcolor}{black!50}
\usepackage[colorlinks,linkcolor=blue,anchorcolor=blue,citecolor=blue]{hyperref}

\SetSymbolFont{letters}{normal}{OML}{txmi}{m}{it}
\SetSymbolFont{operators}{normal}{OT1}{txr}{m}{n}
\SetSymbolFont{largesymbols}{normal}{OMX}{txex}{m}{n}
\SetMathAlphabet{\mathrm}{normal}{OT1}{txr}{m}{n}  
\newtheorem{algo}{Algorithm}
\newtheorem{asmp}{Assumption}
\newtheorem{rem}{Remark}
\newtheorem{prop}{Proposition}
\newtheorem{thm}{Theorem}

\def\cA{{\mathcal A}}

\def\cC{{\mathcal C}}

\def\cF{{\mathcal F}}

\def\cH{{\mathcal H}}
\def\cI{{\mathcal I}}
\def\cJ{{\mathcal J}}
\def\cK{{\mathcal K}}
\def\cL{{\mathcal L}}

\def\cN{{\mathcal N}}

\def\cQ{{\mathcal Q}}

\def\cS{{\mathcal S}}
\def\cT{{\mathcal T}}

\def\cX{{\mathcal X}}
\def\cY{{\mathcal Y}}

\def\bR{{\mathbb R}}
\DeclareMathOperator*{\argmin}{argmin}
\DeclareMathOperator*{\ri}{ri} \DeclareMathOperator*{\dom}{dom}

\def\ds{\displaystyle}
\def\ds{\displaystyle}
\def\[{\begin{equation}}
\def\]{\end{equation}}

\def\baselinestretch{1.42}
\begin{document}

\title{A Three-Operator Splitting Perspective of a Three-Block ADMM for Convex Quadratic Semidefinite Programming and Extensions
}
\titlerunning{A Three-Operator Splitting Perspective of a Three-Block ADMM}
\author{X. K. Chang  \and L. Chen  \and S. Y. Liu}
\institute{
Xiaokai Chang
\at  School of Mathematics and Statistics, Xidian University, Xi’an 710071, P.R. China.\\
\email{xkchang@lut.cn}
\and
Liang Chen
\at College of Mathematics and Econometrics, Hunan University, Changsha 4100082, P.R. China
\\
\email{chl@hnu.edu.cn}
\and
Sanyang Liu
\at
School of Mathematics and Statistics, Xidian University, Xi'an 710071, P.R. China.
\\
\email{liusanyang@126.com}
}
\date{2017/10/23, revised on 2018/07/05}

\maketitle
\begin{abstract}
In recent years, several convergent multi-block variants of the alternating direction method of multipliers (ADMM) have been proposed for solving the convex quadratic semidefinite programming via its dual, which is naturally a $3$-block separable convex optimization problem with one coupled linear equality constraint. Among of these ADMM-type algorithms, the modified $3$-block ADMM in [Chang et al., Neurocomput. 214: 575--586 (2016)] bears a peculiar feature that the augmented Lagrangian function is not necessarily to be minimized with respect to the block-variable corresponding to the quadratic term of the objective function. In this paper, we lay the theoretical foundation of this phenomena by interpreting this modified $3$-block ADMM as a realization of a $3$-operator splitting framework.
Based on this perspective, we are able to extend this modified $3$-block ADMM to a generalized $3$-block ADMM, which not only applies to the more general convex composite quadratic programming setting but also admits the potential of achieving even a better numerical performance.

\end{abstract}

\section{Introduction}
\label{sec_Introduction}
The convex quadratic semidefinite programming (CQSDP) has found many concrete applications in economics and engineering, and captures many well-studied problems, including the linearly constrained semidefinite least squares problem, the nearest Euclidean distance matrix (EDM) problem and the nearest correlation matrix problem.

To solve CQSDP problems, several algorithms have been proposed from different angles and here we only mention a few typical and relevant approaches.
Based on certain perturbations of the Kuhn-Krash-Tucker (KKT) system of the CQSDP problem and its dual, Toh and his coauthors have proposed several inexact interior-point methods for solving them \cite{A polynomial-time,Inexact primal dual,An inexact primal dual path}.
By using the generalized Newton method together with the conjugate gradient method, many efficient methods were proposed for solving CQSDP problems \cite{semis-mooth Newton-CG,augmented Lagrangian dual-H}.
For an important class of the CQSDP problem, i.e. the nearest correlation matrix problem, the quadratic convergence of the Newton-CG method has been obtained by Qi and Sun \cite{semis-mooth Newton-CG}.
For general CQSDP problems, the most recently developed solver QSDPNAL in Li, Sun and Toh \cite{qsdpnal} has demonstrated that a two-phase augmented Lagrangian method, which properly combines both first-order and second-order algorithms, possesses a pretty promising numerical performance.
A closer look of this solver shows that the inexact multi-block proximal ADMM studied in \cite{chenine},  in which the inexact block symmetric Gauss Seidel iteration technique elaborated by Li, Sun and Toh \cite{lisgs} was tightly incorporated, has been utilized to generate an approximate solution with a low to medium accuracy to warm-start an augmented Lagrangian method, whose subproblems are solved by a semismooth Newton method.

Just as in QSDPNAL, first-order ADMM-type algorithms are of their own importance for solving CQSDP problems. In fact, many extensions and modifications of the classic ADMM of Glowinski and Marroco \cite{GLOWINSKI75} and Gabay and Mercier \cite{GABAY76} have been considered in recent years for solving the CQSDP problem via its dual,  which is innately a $3$-block separable convex optimization problem with one coupled linear constraint.
Indeed, the most intuitive idea is to directly extend the classic ADMM to $3$-block problems and the corresponding numerical performance is pretty good for many instances of problems \cite{chang1,chenine,Schur Complement,admm3c}.
However, the direct extension of ADMM (ADMMe) to problems with more than two blocks of variables is not guaranteed to be convergent (c.f. \cite{direct extension} for a concrete example).
Therefore, attentions have been paid to the design of multi-block ADMM-type algorithms and, fortunately, several algorithms of this type have been successfully applied to  solving the CQSDP problem via its dual with a satisfactory numerical efficiency and a theoretical guarantee of convergence \cite{modified admm,chenine,Schur Complement}.

Among of these  ADMM-type algorithms for CQSDP problems, the modified $3$-block ADMM by Chang et al. \cite{modified admm} has a distinct feature that one of the subproblems, i.e., the minimization of the augmented Lagrangian function with respect to a certain block of variables, can always be skipped.
This saves both the computational cost and the memory for variable storage, and, more importantly, the convergence is guaranteed under only one extra condition on the penalty parameter $\sigma$, while the proof for the convergence is much more involved.
The peculiar feature of this method inspired us to get a further understanding of its theoretical foundation.
Moreover, we are also concerned with the question that to what extent can this algorithm being improved or generalized, as well as whether this ADMM-type algorithm can be applied to more general problems.

In order to conduct the theoretical analysis to address the concerns  mentioned above, an indispensable tool is the concept of the maximal monotone operator as well as the corresponding operator  splitting methods \cite{DRS,forward-backward-forward splitting} which are designed to find its roots.
The interconnection between the operator splitting methods and ADMM-type algorithms was first established by Gabay \cite{GABAY83}, who showed that the classic ADMM with unit step-length can be explained as the well-known Douglas-Rachford operator splitting method.
Based on this observation, Eckstein and Bertsekas \cite{ECKSTEIN92} presented a generalized ADMM for the purpose of improving the performance of the classic ADMM (with unit step-length) via an over-relaxation step.
We mention that for a recent survey on this topic one may refer to \cite{eck12}, and one also can refer to \cite{xiao} and references therein for more details and recent progresses on generalized ADMM.
Consequently, we are interested if one can also interpret the modified $3$-block ADMM in \cite{modified admm} via a certain operator splitting scheme and get further improvements on this algorithm via certain over-relaxation steps.

In this paper, we fulfil our objective  by showing that the modified $3$-block ADMM in \cite{modified admm} can be explained as an application of the $3$-operator splitting framework studied in Davis and Yin \cite{Three-Operator}.
We conduct our analysis in a much general setting in which the model that we will consider contains the CQSDP problem as a special case.
Moreover, based on this operator splitting perspective, we present a generalized $3$-block ADMM, in the sense of Eckstein and Bertsekas \cite{ECKSTEIN92},
in which an over-relaxation step is incorporated.
We mention that, such as in Xiao et al. \cite{xiao}, this kind of over-relaxation can lead to an obvious improvement on the numerical efficiency of ADMM-type algorithms.

The remaining parts of this paper are organized as follows.
In Section \ref{prel}, we give a quick review of the CQSDP problem and the modified 3-block ADMM algorithm proposed by Chang et al. \cite{modified admm}.
Section \ref{sec_Three-Operator} is devoted to the operator-splitting perspective of this modified $3$-block ADMM for the CQSDP problem.
In Section \ref{main}, we introduce the convex composite quadratic optimization model and present a generalized version of the modified $3$-block ADMM in \cite{modified admm} for solving this problem.
With the result established in Section \ref{sec_Three-Operator}, the convergence analysis of the proposed algorithm can be conducted in a very concise manner.
We conclude this paper in Section \ref{conclusion}.

\section{Preliminaries}
\label{prel}

\subsection{Basic Concepts}
Let $\cH$ be an arbitrary finite dimensional real Hilbert space endowed with an inner product denoted by $\langle \cdot, \cdot\rangle$ and its induced norm $\|\cdot\|$.
Let $\cF:\cH\to\cH$ be an arbitrary set-valued mapping.
If $\cF$ is single-valued, it is called  $\beta$-cocoercive (or $\beta$-inverse-strongly monotone)  for a certain constant $\beta>0$, if
$$\langle {\mathcal{F}}(x)-{\mathcal{F}}(x'), x-x'\rangle\geq \beta \|{\mathcal{F}}x- {\mathcal{F}}x'\|^2,\quad
\forall x, x'\in \cH;
$$
If $\cF$ is a self-adjoint positive semidefinite linear operator, we use $\lambda_{\max}(\cF)$ to denote its largest eigenvalue, i.e. $\lambda_{\max}(\cF):=\max_{\|x\|=1}\langle x, \cF x\rangle$.  In this case, it is easy to verify that
$\cF$ is $\frac{1}{\lambda_{\max}(\cF)}$-coercive.
If $\cF$ is a multi-valued maximal monotone operator and $\sigma>0$ is a constant,
the mapping $\cJ_{\sigma\cF}:=(\cI+\sigma\cF)^{-1}$, which is called the Minty resolvant of $\cF$,  is a single valued  mapping, and this mapping is also nonexpansive \cite[Theorem 12.12]{rocva}. Here, $\cI$ denotes the identity operator from $\cH$ to itself and it will be kept as the notation for the identity operator from any space to itself, if no ambiguity is caused.

Let $f:\cH\to(-\infty,\infty]$ be a closed proper convex function. The subdifferential mapping $\partial f$ of $f$ is then a maximal monotone operator and in this case
\begin{eqnarray*}
{\mathcal{J}}_{\sigma \partial f}(x)
= (\cI+\sigma\partial f)^{-1}(x)
= \argmin_z\left\{ f(z)+\frac{1}{2\sigma}\|x-z\|^2\right\},\quad\forall\, x\in\cH.
\end{eqnarray*}
For any set $\cC\subset \cH$, the indicator function $\delta_\cC:\cH\to(-\infty,\infty]$ is defined by $\delta_\cC(x)=0$ if $x\in\cC$ and $\delta_\cC(x)=+\infty$ otherwise.
If $\cC$ is a closed convex set, $\delta_\cC$ is therefore a closed proper convex function and, in this case,
$
{\mathcal{J}}_{\sigma \partial \delta_\cC}(x)=\Pi_\cC(x)
$, i.e., the metric projection of $x$ onto $\cC$,
and
$$
\partial\delta_\cC(x)=\cN_{\cC}(x):=\{z\,|\, \langle z, x'-x\rangle\le 0,\ \forall x'\in\cH\}.
$$
Here, the mapping $\cN_\cC$ is called the normal cone mapping of the set $\cC$ and $\cN_c(x)$ is called the normal cone of $\cC$ at $x$, which is a closed convex cone.

Let ${\mathcal{S}}^n$ be the space of $n\times n$ real symmetric matrices endowed with the standard trace inner product $\langle \cdot,\cdot \rangle$ and the Frobenius norm $\|\cdot\|$. We use ${\mathcal{S}}_{+}^{n}$ and ${\mathcal{S}}_{++}^{n}$ to denote the sets of symmetric positive semidefinite and positive definite matrices in ${\mathcal{S}}^{n}$, respectively.

\subsection{The CQSDP Problem}
The CQSDP problem takes the following standard form:

\begin{eqnarray}
\label{CQSDP}
\min_X \;\; && \frac{1}{2} \left\langle X,\varphi (X)\right\rangle +\langle C , X\rangle \nonumber\\
\mbox{s.t.} \;\;&& {\mathcal{A}}(X)= b ,\  X  \in {\mathcal{S}}^n_+,
\end{eqnarray}
where $\varphi: {\mathcal{S}}^n \rightarrow
{\mathcal{S}}^n$  is a given self-adjoint positive semidefinite linear operator,
${\mathcal{A}}: {\mathcal{S}}^n \rightarrow {\mathbb{R}}^m $  is a linear map defined by
\\
$$
{\mathcal{A}}(X):=\left(
\begin{array}{c}
\langle A_1 ,X\rangle \\
\vdots  \\
\langle A_m ,X\rangle
\end{array}
\right)\in {\mathbb{R}}^m,
\quad\forall X\in\cS^n
$$
with the given $A_i \in {\mathcal{S}}^n$, $i=1,\ldots,m$, and $ b \in {\mathbb{R}}^m$.
The adjoint of ${\mathcal{A}}$, with respect to the standard inner products in ${\mathcal{S}}^n$  and $ {\mathbb{R}}^m$, is denoted by  ${\mathcal{A}}^*$.

Note that ${\mathcal{S}}_{+}^{n}$ is a closed convex self-dual cone. Then, the dual of problem (\ref{CQSDP}) can be equivalently formulated in minimization form as follows
\begin{eqnarray}
\label{CQSDPD}
 \min_{W,y,Z} \;\;&&\frac{1}{2} \langle W,\varphi (W)\rangle -b^Ty+\delta_{{\mathcal{S}}^n_+}(Z) \nonumber\\
\mbox{s.t.} \;\; && -\varphi (W)+{\mathcal{A}}^*(y)+Z=C,
\end{eqnarray}
where $\delta_{\cS_+^n}$ is the indicator function of $\cS_+^n$, $W\in\cS^n$, $Z\in\cS^n$ and $y\in\bR^m$.
The Lagrangian function of problem \eqref{CQSDPD} is defined by
$$
\begin{array}{r}
l(W,y,Z;X):=\frac{1}{2}\langle W,\varphi (W)\rangle -b^Ty+\delta_{{\mathcal{S}}^n_+}(Z)+\langle -\varphi (W)+{\mathcal{A}}^*(y)+Z-C, X\rangle,\quad
\\[1mm]
\forall(W,y,Z,X)\in\cS^n\times\bR^m\times\cS^n\times\cS^n.
\end{array}
$$
Therefore, the KKT system of problem \eqref{CQSDPD} is given by
\begin{eqnarray}\label{KKT}
\left\{
\begin{array}{l}
\varphi(W)=\varphi(X),\
{\mathcal{A}}(X)=b,\\[1mm]
~-\varphi (W)+{\mathcal{A}}^*(y)+Z=C,\\[1mm]
~X\in {\mathcal{S}}^n_+,~Z\in {\mathcal{S}}^n_+,~\langle Z,X\rangle=0.
\end{array}\right.
\end{eqnarray}
For any $(W,y,Z,X)\in\cS^n\times\bR^m\times\cS^n\times\cS^n$ satisfying the KKT system \eqref{KKT}, $(W,y,Z)$ is a solution to problem \eqref{CQSDPD} while $X$ is a solution to problem \eqref{CQSDP}.

\subsection{A Modified $3$-Block ADMM for Problem (\ref{CQSDPD})}
\label{sec_Modified ADMM}
Let $\sigma>0$ be the penalty parameter. The augmented Lagrangian function of  problem (\ref{CQSDPD}) can be defined by
\[
\label{auglag}
\begin{array}{l}
\cL_\sigma(W,y,Z;X):=l(W,y,Z;X)+\frac{\sigma}{2}\|-\varphi (W)+{\mathcal{A}}^*(y)+Z-C\|^2,
\\[1mm]
\hspace{5cm}\forall(W,y,Z,X)\in\cS^n\times\bR^m\times\cS^n\times\cS^n.
\end{array}
\]
Choose an initial points $(y^0,Z^0,X^0)\in{\mathbb{R}}^m\times{\mathcal{S}}^n_+\times\in {\mathcal{S}}^n$.
A direct extension of the classic ADMM to the $3$-block problem (\ref{CQSDPD}) consists of the following steps, for $k=0,1,\ldots$,
\[
\left\{
\label{ADPsi}
\begin{array}{lcl}
W^{k+1}&:=&\ds\argmin_{W } {\mathcal{L}}_\sigma(W,y^{k},Z^{k};X^k), \\[1mm]
y^{k+1}&:=&\ds\argmin_{y} {\mathcal{L}}_\sigma(W^{k+1},y,Z^k;X^k), \\[1mm]
Z^{k+1}&:=&\ds\argmin_{Z } {\mathcal{L}}_\sigma(W^{k+1},y^{k+1},Z;X^k), \\[1mm]
X^{k+1}&:=&X^k+\tau\sigma\left({\mathcal{A}}^*(y^{k+1})+Z^{k+1}-\varphi (W^{k+1})- C\right),
\end{array}
\right.
\]
where $\tau> 0$ is the step-length. Generally, the convergence of the iteration scheme (\ref{ADPsi}) can not be guaranteed.

In Chang et al. \cite{modified admm}, the authors have proposed the following algorithm to solve  problem \eqref{CQSDPD}, by modifying the iteration scheme (\ref{ADPsi}).

\medskip
\centerline{\fbox{\parbox{0.97\textwidth}{
\begin{algo}[A Modified ADMM for the CQSDP problem (\ref{CQSDPD})]
\label{algo1}
\end{algo}
Let $\sigma>0$ be the given parameter. Choose $Z^0\in {\mathcal{S}}^n_+$ and $X^0\in {\mathcal{S}}^n$. For $k=0,1,\ldots,$
\begin{description}
\item[{\bf Step 1.}] Set $W^{k+1}:=X^{k}$;
\item[{\bf Step 2.}] Compute $y^{k+1}:=\argmin\limits_{y} {\mathcal{L}}_\sigma(W^{k+1},y,Z^k;X^k)$;
\item[{\bf Step 3.}] Compute $Z^{k+1}:=\argmin\limits_{Z} {\mathcal{L}}_\sigma(W^{k+1},y^{k+1},Z;X^k)$;
\item[{\bf Step 4.}] Update $X^{k+1}:=X^k+\sigma\left({\mathcal{A}}^*(y^{k+1})+Z^{k+1}-\varphi (W^{k+1})- C\right)$.
\end{description}
}}}

\medskip

Compared with the directly extended $3$-block ADMM scheme \eqref{ADPsi}, Algorithm \ref{algo1} always set $W^{k+1}$ as $X^k$ instead of minimizing the augmented Lagrangian function with respect to the block-variable $W$.
The original idea of this $3$-block ADMM is quite intuitive since that for any $(W,y,Z,X)\in\cS^n\times\bR^m\times\cS^n\times\cS^n$ being a solution to the KKT system of the CQSDP problem, $(X,y,Z,X)$ is also a solution to it.
Therefore, whenever dealing with a subproblem with respect to $W$, one may directly use the value of $X$ to substitute $W$ instead of solving this subproblem. However, the convergence analysis of Algorithm \ref{algo1} in \cite{modified admm} is very complicated.

The following Assumption was used
in \cite{modified admm} for
analyzing the convergence of Algorithm \ref{algo1}.
\begin{asmp}
\label{assum1}
The linear operator $\cA$ is surjective
and the Slater's constraint qualification holds for problem \eqref{CQSDP}, i.e. there exists a matrix $\tilde{X} \in {\mathcal{S}}^n_{++}$ satisfying ${\mathcal{A}}( \tilde{X} )= b$.
\end{asmp}
\begin{rem}
The first part of Assumption \ref{assum1} implies that the linear operator ${\mathcal{A}}{\mathcal{A}}^*$ is nonsingular.
Consequently, Step 2 of Algorithm \ref{algo1} is well-defined.
 Moreover, under Assumption \ref{assum1}, we know from \cite[Corollaries 28.2.2 \& 28.3.1]{rocbook} that $X\in\cS^n$ is a solution to problem \eqref{CQSDP} if and only if there exists a vector $(W,y,Z)\in\cS^n\times\bR^m\times\cS^n$ such that $(W,y,Z,X)$ is a solution to the KKT system \eqref{KKT}.
Additionally, for any $(W,y,Z,X)$ that satisfies \eqref{KKT}, from \cite[Corollary 30.5.1]{rocbook} we know that $X$ is an optimal solution to problem  \eqref{CQSDP} and  $(W,y,Z)$ is an optimal solution to problem \eqref{CQSDPD}.
\end{rem}
The convergence properties of Algorithm \ref{algo1} have been given in \cite[Theorem 1]{modified admm}. We summarize these results as follows.

\begin{prop}
\label{propconv}
Suppose that the solutions set to problem \eqref{CQSDP} is nonempty, Assumption \ref{assum1} holds and  $\sigma\in \left(0,\frac{1}{\lambda_{\max}(\varphi)}\right]$.  Then, the sequence $\{(W^k,y^k,Z^k, X^k)\}$ generated by Algorithm \ref{algo1} converges to a point which is a solution to the KKT system \eqref{KKT}.
\end{prop}

\begin{rem}
If $\varphi=0$, problem (\ref{CQSDP}) is then a standard linear semidefinite programming problem and its dual will be a 2-block convex optimization problem.
In this case, the requirement that $\sigma\in\left(0,\frac{1}{\lambda_{\max}(\varphi)}\right]$ is no longer necessary and Algorithm \ref{algo1} is then automatically the classic 2-block ADMM with the unit step-length.
\end{rem}

\section{A Three-Operator Splitting Perspective}
\label{sec_Three-Operator}
In this  section, we narrate  Algorithm \ref{algo1} from a $3$-operator splitting perspective.  Note that the problem \eqref{CQSDP} can be written as
$$
\min_{X\in\cS^n} f(X) + g(X) + h(X),
$$
where
$$
\left\{
\begin{array}{l}
\ds
f(X):= \delta_{K}(X)\quad\mbox{with}\quad K:=\{X\in\cS^n |\cA(X)=b\},
\\[2mm]
\ds
g(X):=\delta_{\cS^n_+}(X),
\\[2mm]
\ds
h(X):=\frac{1}{2}\langle X,\varphi(X)\rangle+\langle C, X\rangle.
\end{array}
\right.
$$
Under Assumption \ref{assum1}, we know from \cite[Theorem 23.8]{rocbook} that $X$ is a solution to problem \eqref{CQSDP} if and only if
$$
0\in\cN_{K}(X)+\cN_{\cS_+}(X) +\left( \varphi(X)+C\right) = \partial f(X) + \partial g(X) +\nabla  h (X).
$$
Therefore, one can try to solve problem \eqref{CQSDP} via solving the above inclusion problem.
In fact, Algorithm \ref{algo1} can be interpreted as an operator splitting algorithm applied to solve this inclusion problem.
This will be explained in details as follows.

Let $\{X^k\}$, $\{y^k\}$ and $\{Z^k\}$ be the sequences generated by Algorithm \ref{algo1}. We define for $k\ge 0$,
\[
\label{defuk}
U^{k}:=X^k-\sigma(\varphi(X^k)+C-\cA^*y^{k+1}).
\]
Moreover, just as \cite[Equation (1.2)]{Three-Operator}, we define the mapping $\cT:\cS^n\to\cS^n$ by
$$
\cT:=\cI-\cJ_{\sigma\cN_{\cS_+^n}}
+\cJ_{\sigma\cN_{K}}\circ\big(2\cJ_{\sigma\cN_{\cS_+^n}}-\cI-\sigma\nabla  h (\cJ_{\sigma\cN_{\cS_+^n}})\big).
$$
Then, we have the following result.
\begin{thm}
\label{thmmain}
Let $\{X^k\}$, $\{y^k\}$ and $\{Z^k\}$ be the sequences generated by Algorithm \ref{algo1}. Then the sequence $\{U^k\}$ defined in \eqref{defuk} satisfies
$$
U^{k+1}=\cT(U^k).
$$
\end{thm}
\begin{proof}
Note that for any $k\ge 0$, steps $1$ to $4$ of Algorithm \ref{algo1} can be reorganized as follows
\begin{equation}
\left\{
\label{yk}
\begin{array}{l}
y^{k+1}:=-(\cA\cA^*)^{-1}((\cA X^k-b)/\sigma+\cA(Z^k-\varphi(X^k)-C)),
\\[2mm]
X^{k+1}:=\Pi_{\cS^n_+}(U^{k}),
\\[2mm]
Z^{k+1}:=(X^{k+1}-U^{k})/\sigma.
\end{array}
\right.
\end{equation}
Therefore, it holds that
\[
\label{xplus}
X^{k+1}=(\cI+\sigma \cN_{\cS^n_+})^{-1}(U^k)=\cJ_{\sigma\cN_{\cS_+^n}}(U^k).
\]
Moreover, one can readily obtain that
$$
\begin{array}{ll}
y^{k+2}&=-\frac{1}{\sigma}(\cA\cA^*)^{-1}( \cA X^{k+1}-b +\sigma\cA(Z^{k+1}-\varphi(X^{k+1})-C))
\\[2mm]
&=-\frac{1}{\sigma}(\cA\cA^*)^{-1}\left( \cA X^{k+1}-b +\sigma\cA\left[\frac{1}{\sigma}(X^{k+1}-U^k)-\varphi(X^{k+1})-C\right]\right)
\\[2mm]
&=-\frac{1}{\sigma}(\cA\cA^*)^{-1}\left( \cA X^{k+1}-b +\cA (X^{k+1}-U^k)-\sigma\cA\left[\varphi(X^{k+1})+C\right]\right)
\\[2mm]
&
=-\frac{1}{\sigma}(\cA\cA^*)^{-1}\left( \cA \left(2X^{k+1}- U^k-\sigma\big[\varphi(X^{k+1})+C\big]\right)-b\right).
\end{array}
$$
Note that
$$\Pi_{K}(X)=X-\cA^*(\cA\cA^*)^{-1}(\cA X-b),
\quad\forall X\in\cS^n. $$
Then, by using \eqref{xplus} we can get
$$
\begin{array}{l}
\left(\cJ_{\sigma\cN_{K}}\circ\big(2\cJ_{\sigma\cN_{\cS_+^n}}-\cI-\sigma\nabla  h (\cJ_{\sigma\cN_{\cS_+^n}})\big)\right)(U^k)
\\[2mm]
=
\Pi_{K}\left(2X^{k+1}- U^k-\sigma\big[\varphi(X^{k+1})+C\big]\right)
\\[2mm]
=\left(2X^{k+1}- U^k-\sigma\big[\varphi(X^{k+1})+C\big]\right)+\sigma\cA^*y^{k+2}.
\end{array}
$$
Moreover, it is easy to see from \eqref{defuk} that
$$
\begin{array}{ll}
U^{k+1}
&
=X^{k+1}-\sigma(\varphi(X^{k+1})+C-\cA^*y^{k+2})
\\[2mm]
&
=X^{k+1}+\sigma\cA^*y^{k+2}-\sigma\big[\varphi(X^{k+1})+C\big ]
\\[2mm]
&=U^{k}+\left(X^{k+1}- U^k-\sigma\big[\varphi(X^{k+1})+C\big ]\right)+\sigma\cA^*y^{k+2}
\\[2mm]
&=U^{k}-X^{k+1}+\left(2X^{k+1}- U^k-\sigma\big[\varphi(X^{k+1})+C\big ]\right)+\sigma\cA^*y^{k+2}
\\[2mm]
&=U^{k}-\cJ_{\sigma\cN_{\cS_+^n}}(U^k)+
\left(\cJ_{\sigma\cN_{K}}\circ\big(2\cJ_{\sigma\cN_{\cS_+^n}}-\cI-\sigma\nabla  h (\cJ_{\sigma\cN_{\cS_+^n}})\big)\right)(U^k),
\end{array}
$$
which, together with the definition of $\cT$, completes the proof.
\qed
\end{proof}
\begin{rem}
The definition of the operator $\cT$ was introduced by Davis and Yin \cite[Equation (1.2)]{Three-Operator}. This operator was regarded as a combination of the well-known Douglas-Rachford splitting and the forward-backward splitting.
As will be seen in the next section, based on the properties of $\cT$, the global convergence of Algorithm \ref{algo1}, with the looser requirement $\sigma\in (0, \frac{2}{\lambda_{\max}(\varphi)})$, can be alternatively proved by using \cite[Theorem 2.1]{Three-Operator} together with Theorem \ref{thmmain}.
 \end{rem}

\section{Generalizations and Extensions}
\label{main}
The successful application of Algorithm \ref{algo1} to the CQSDP problem \eqref{CQSDP} via its dual \eqref{CQSDPD} and the explanation from the operator splitting perspective made in Section
\ref{sec_Three-Operator} inspired us to consider extending this algorithm to much general problems.
In this section, we consider the following convex composite quadratic programming \cite{Schur Complement} problem
\begin{eqnarray}
\label{cqp}
\min_{x\in\cX} \;\; && \theta^*(x)+\frac{1}{2} \left\langle x,\cQ (x)\right\rangle +\langle c , x\rangle \nonumber\\
\mbox{s.t.} \;\;&& {\mathcal{A}} x = b,
\end{eqnarray}
where $\cA:\cX\to\cY$ is a linear map, $\cX$ and $\cY$ are finite dimensional Euclidean spaces each endowed with a inner product $\langle\cdot,\cdot\rangle$ and its induced norm $\|\cdot\|$. $\theta^*$ is the Fenchel conjugate function of the closed proper convex (possibly nonsmooth) function $\theta:\cX\to(-\infty,\infty]$, $\cQ:\cX\to\cX$ is a self-adjoint positive semidefinite linear operator, and $c\in\cX$ and $b\in\cY$ are the given data.
Obviously, problem (\ref{CQSDPD}) is an instance of problem (\ref{cqpd})
 in which $\cX=\cS^n$, $\cY=\mathbb{R}^m$ and $\theta^*$ being the indicator function of $\cS^n_+$.
We make the following assumption on problem \eqref{cqp}.
\begin{asmp}
\label{ass2}
The linear operator $\cA$ is surjective and there exists a point $x\in\ri(\dom\theta^*)$ such that $\cA x=b$.
\end{asmp}
Under Assumption \ref{ass2} we know that $x$ is a solution to problem \eqref{cqp} if and only if there exists a vector $(w,y,z)\in\cX\times\cY\times\cX$ such that $(x,w,y,z)$ solves the following KKT system of problem \eqref{cqpd}
\begin{eqnarray}
\label{KKT2}
\left\{
\begin{array}{l}
\cQ w=\cQ x,\\[1mm]
\cA x-b=0,\\[1mm]
0\in x-\partial \theta(-z),\\[1mm]
\cA^*y+z-\cQ w-c=0.
\end{array}\right.
\end{eqnarray}
Moreover, such a vector $(w,y,z)$ is a solution to the dual of problem \eqref{cqp}, which can equivalently be recast in minimization form as
\[
\label{cqpd}
\min_{w,y,z}\left\{ \frac{1}{2}\langle w,\cQ w\rangle -\langle b,y\rangle+\theta(-z)
\ |\
-\cQ w+\cA^*y+z=c
\right\},
\]
where the decision variables $w\in\cX,y\in\cY$ and $z\in\cX$.

Let $\sigma>0$ be the penalty parameter. The augmented Lagrangian function of problem \eqref{cqpd} is defined by
$$
\begin{array}{ll}
\ds
\cL_{\sigma}(w,y,z;x):=&\frac{1}{2}\langle w,\cQ w\rangle -\langle b,y\rangle+\theta(-z)
 +\langle \cA^*y+z-\cQ w-c,x\rangle+\frac{\sigma}{2}\|\cA^*y+z-\cQ w-c\|^2,\\[2mm]
&\hfill\forall(w,y,z;x)\in\cX\times\cY\times\cX\times\cX.
\end{array}
$$
In sequel, we will extend and generalize Algorithm \ref{algo1} to problem (\ref{cqpd}), and prove its convergence via the existing convergence theorem of the 3-operator splitting method in \cite{Three-Operator}.

\medskip
\centerline{\fbox{\parbox{0.97\textwidth}{
\begin{algo}[A Generalized Modified ADMM for problem (\ref{cqpd})]\label{algo2}
\end{algo}
Let $\sigma>0$ and $\rho\in(0,2)$. Choose initial variables $z^0$ such that $-z^0\in\dom\theta$ and $x^0\in\cX$. For $k=0,1,\ldots,$
\begin{description}
\item[{\bf Step 1.}] Set $w^{k+1}:=x^{k}$;
\item[{\bf Step 2.}] Compute $y^{k+1}:=\argmin\limits_{y} {\mathcal{L}}_\sigma(w^{k+1},y,z^k;x^k)$;
\item[{\bf Step 3.}]
Compute
$$
\begin{array}{rl}
\ds  z^{k+1}:=\argmin_{z}&\left\{
 \theta(-z)+\langle z,x^k\rangle
  +\frac{\sigma}{2}\|\rho\cA^*y^{k+1}-(1-\rho)z^k+z-\rho\cQ w^{k+1}-\rho c\|^2\right\};\\
\end{array}
$$
\item[{\bf Step 4.}] Update $x^{k+1}:=x^k+
\sigma (\rho\cA^*y^{k+1}-(1-\rho)z^k+z^{k+1}-\rho\cQ w^{k+1}-\rho c).$
\end{description}
}}}

\medskip
\begin{rem}
The above algorithm is called a generalized modified ADMM since that it can viewed as a direct extension of the generalized $2$-block ADMM \cite{ECKSTEIN92} to problem \eqref{cqpd}.  Moreover, we should mention that, generally, since that $\cA$ is surjective, all the subproblems are well-defined and admit unique solutions.
The well-definedness of subproblems is very essential for ADMM-type algorithm. On this part, one may refer to a counterexample by Chen et al. \cite[Section 3]{admmnote}.
\end{rem}

\begin{rem}
For the case that $\rho=1$, step $3$ of Algorithm \ref{algo2} turns to
$$
z^{k+1}=\argmin_z\cL_{\sigma}(w^{k+1},y^{k+1},z;x^k).
$$
In this case, the direct extension of the generalized ADMM is then a direct extension of the classic ADMM with unit step-length, whose $k$-th step takes the following form
\[
\label{admm3e}
\left
\{\begin{array}{l}
w^{k+1}\in\argmin_{w} {\mathcal{L}}_\sigma(w,y^{k},z^{k};x^k),\\[2mm]
y^{k+1}\in\argmin_{y} {\mathcal{L}}_\sigma(w^{k+1},y,z^k;x^k),\\[2mm]
z^{k+1}\in\argmin_{z} {\mathcal{L}}_\sigma(w^{k+1},y^{k+1},z;x^k),\\[2mm]
x^{k+1}=x^k+\sigma ( \cA^*y^{k+1}+z^{k+1}-\cQ w^{k+1}- c).
\end{array}
\right.
\]
If the order of solving the subproblems is further changed as follows
$$
\left
\{\begin{array}{l}
y^{k+1}\in\argmin_{y} {\mathcal{L}}_\sigma(w^{k},y,z^k;x^k),\\[2mm]
w^{k+1}\in\argmin_{w} {\mathcal{L}}_\sigma(w,y^{k+1},z^{k};x^k),\\[2mm]
z^{k+1}\in\argmin_{z} {\mathcal{L}}_\sigma(w^{k+1},y^{k+1},z;x^k),\\[2mm]
\end{array}
\right.
$$
the convergence of this direct extension of the class ADMM has been established in Li et al. \cite[Theorem 2.1]{limin} under certain conditions\footnote{
In \cite{limin}, the authors also have considered adding proximal terms to subproblems and using a dual step-size which can be chosen in $\big(0,(1+\sqrt{5})/2\big)$.
Since that one can restrict $w$ always in the range space of the linear operator $\cQ$ so that $f_1(w)=\frac{1}{2}\langle w,\cQ w\rangle$ is a strongly convex function. Hence, the results in \cite{limin} are applicable.}.
In \cite[Section 4.2]{Three-Operator}, the authors have considered another $3$-block extension of the classic ADMM, i.e., \cite[Algorithm 7]{Three-Operator}.
The difference of this extension from \eqref{admm3e} is that the subproblem for computing  $w^{k+1}$ does not contain the penalty term $\frac{\sigma}{2}\|\cA^*y^k+z^k-\cQ w-c\|$.
Moreover, in the corresponding convergence analysis, it requires $\sigma\in\left(0,\frac{2\lambda^+_{\min}(\cQ)}{(\lambda_{\max} (\cQ))^2}\right)$, where $\lambda^+_{\min}(\cQ)$ denotes the smallest positive eigenvalue of $\cQ$.
This requirement of $\sigma$ is obviously stronger than the condition that $\sigma\in \left(0,\frac{2}{\lambda_{\max}(\cQ)}\right)$, which will be used in the forthcoming convergence analysis of Algorithm \ref{algo2}.
\end{rem}

Next, we analyze the convergence properties of Algorithm \ref{algo2}.
Suppose that $\{w^k\}$, $\{y^k\}$, $\{z^k\}$ and $\{x^k\}$ be the infinite sequences generated by Algorithm \ref{algo2}. Define for $k\ge 0$
\[
\label{sequk}
u^{k}:=x^k+
\sigma (\rho\cA^*y^{k+1}-(1-\rho)z^k -\rho\cQ w^{k+1}-\rho c)
=x^{k+1}-\sigma z^{k+1}.
\]
For convenience, we define the convex set
$$
\cK:=\{x\in\cX|\cA x=b\},
$$
and the
quadratic function $q:\cX\to(-\infty,\infty)$ by
$$
q(x):=\frac{1}{2}\langle x,\cQ x\rangle+\langle c,x\rangle,\quad \forall x\in\cX.
$$
Then, the gradient of the function $q$ is given by
$\nabla q(x)=\cQ x+c$, $\forall x\in\cX$.
Moreover, we define a single-valued mapping $\Gamma:\cX\to\cX$ by
\[
\label{mainoper}
\Gamma:=\cI-\cJ_{\sigma\partial \theta^*}+
\cJ_{\sigma\cN_\cK}\circ\left(2 \cJ_{\sigma\partial\theta^*}-\cI-\sigma \nabla q\circ \cJ_{\sigma\partial\theta^*}\right) .
\]
Based on the above definitions we have the following result.

\begin{prop}
\label{prop2}
Suppose that $\{w^k\}$, $\{y^k\}$ and $\{z^k\}$ are the infinite sequences generated by Algorithm \ref{algo2}, and $\{u^k\}$ is the sequence defined by \eqref{sequk}.
Then, one has that
$$
u^{k+1}=(1-\rho)u^k+\rho\Gamma(u^{k}).
$$
\end{prop}

\begin{proof}
Note that for any $k\ge 0$
\[
\label{optcond1}
\begin{array}{rl}
0&\in-\partial \theta(-z^{k+1})+ x^k+\sigma (\rho\cA^*y^{k+1}-(1-\rho)z^k+z^{k+1}-\rho\cQ w^{k+1}-\rho c)
\\[2mm]
&=-\partial \theta(-z^{k+1})+ x^{k+1}.
\end{array}
\]
Since that $\theta$ is a closed proper convex function, by using \cite[Theorem 23.5]{rocbook} we have that
$
x^{k+1}\in \partial \theta(-z^{k+1})
$
so that
$
-z^{k+1}\in\partial \theta^*( x^{k+1})
$.
Therefore, it holds that
$$
\begin{array}{ll}
0\in
& \partial \theta^*( x^{k+1})+   z^{k+1}
= \partial \theta^*( x^{k+1})
+\frac{1}{\sigma}\left(x^{k+1}-(x^{k+1}-\sigma  z^{k+1})\right)
\\[2mm]
&= \partial \theta^*( x^{k+1})
+\frac{1}{\sigma} (x^{k+1}-u^k),
\end{array}
$$
where we have used the fact that $u^k=x^{k+1}-\sigma  z^{k+1}$ from \eqref{sequk}.
Thus, by using \cite[Theorem 23.8 \& 23.9]{rocbook} and the above inclusion one can get that
\[
\label{pf21}
x^{k+1}=
\argmin_{x}\left\{ \theta^*( x)+\frac{1}{2\sigma}\|x-u^k\|^2\right\}
=\left(\cI+\sigma\partial \theta^*\right)^{-1}(u^k)
=\cJ_{\sigma\partial\theta^*}(u^k).
\]
On the other hand, one can readily obtain that
\[
\label{optt}
0=-b+\cA x^{k+1}
+\sigma\cA(\cA^*y^{k+2}+z^{k+1}-\cQ w^{k+2}-c).
\]
Therefore,
\[
\label{ykp2}
\begin{array}{ll}
y^{k+2}&=-[\sigma\cA\cA^*]^{-1}\left((\cA x^{k+1}-b)
+\sigma \cA(z^{k+1}-\cQ x^{k+1}-c)\right)
\\[2mm]
&=-[\sigma\cA\cA^*]^{-1}\left( \cA (2x^{k+1}-u^k-\sigma(\cQ x^{k+1}+c) )-b\right).
\end{array}
\]
Note that for any $\xi\in\cX$ one has $\Pi_{\cK}(\xi)=\xi-\cA^*(\cA\cA^*)^{-1}(\cA \xi-b)$.
Consequently, by using \eqref{pf21} and \eqref{ykp2} we can get that
$$
\begin{array}{l}
 \cJ_{\sigma\cN_{\cK}}\left( \big(2\cJ_{\sigma\partial\theta^*}-\cI-\sigma\nabla  f\circ  \cJ_{\sigma\partial\theta^*}\big) (u^k)\right)
\\[2mm]
=
\Pi_{\cK}\left(2x^{k+1}-u^k-\sigma(\cQ x^{k+1}+c )
\right)
\\[2mm]
=2x^{k+1}-u^k-\sigma(\cQ x^{k+1}+c)
+\sigma\cA^*y^{k+2}.
\end{array}
$$
From \eqref{sequk} and the fact that $w^{k+1}=x^{k}$ one has that
$$
\begin{array}{ll}
u^{k+1}
&
=x^{k+1}+\sigma(\rho\cA^*y^{k+2}-(1-\rho)z^{k+1} -\rho\cQ x^{k+1}-\rho c)
\\[2mm]
&
=u^k-\rho x^{k+1}+(1+\rho)x^{k+1}-u^k+\sigma\rho\cA^*y^{k+2}
\\[1mm]
&\quad\quad
-\sigma(1-\rho)z^{k+1}
-\sigma\rho(\cQ x^{k+1}+ c)
\\[2mm]
&
=u^k-\rho x^{k+1}
+\rho(
2x^{k+1}
- u^k
-\sigma (\cQ x^{k+1}+ c)+\sigma \cA^*y^{k+2})
\\[1mm]
&\quad\quad
+(1-\rho)(x^{k+1}
-u^k
-\sigma z^{k+1}).
\end{array}
$$
Note that \eqref{sequk} tells that $u^{k}=x^{k+1}-\sigma z^{k+1}$.
Therefore, we can readily get
$$
\begin{array}{ll}
u^{k+1}
&=u^k-\rho x^{k+1}
+\rho(
2x^{k+1}
- u^k
-\sigma (\cQ x^{k+1}+ c)
+\sigma \cA^*y^{k+2})
\\[2mm]
&=(1-\rho)u^{k}+\rho
\left(u^k-
x^{k+1}
+
2x^{k+1}
- u^k
-\sigma (\cQ x^{k+1}+ c)
+\sigma \cA^*y^{k+2}\right)\\[2mm]
&=(1-\rho)u^{k}+\rho\Gamma(u^k),
\end{array}
$$
which completes the proof.
\qed
\end{proof}

According to Proposition \ref{prop2},  Algorithm \ref{algo2} can also be viewed as a realization of the 3-operator splitting scheme proposed in \cite{Three-Operator} applied to the following problem
\[
\label{3blkdual}
\min_x~~\left\{\theta^*(x)+\delta_\cK(x)+\frac{1}{2} \langle x,\cQ x\rangle+\langle c,x\rangle\right\}.
\]
Therefore, by using Proposition \ref{prop2}, a part of the convergence properties of Algorithm \ref{algo2} can be deduced directly from \cite[Theorem 1.1]{Three-Operator}.  We summarize it as follows.
\begin{prop}
\label{thm:conv1}
Suppose that the solution set to problem \eqref{3blkdual} is nonempty and Assumption \ref{ass2} holds.
Let the infinite sequences
$\{w^k\}$, $\{y^k\}$, $\{z^k\}$ and $\{x^k\}$ be generated by Algorithm \ref{algo2} with
$\sigma\in \left(0,\frac{2}{\lambda_{\max}(\cQ)}\right)$
and
$\rho\in\left(0,\frac{4-\sigma\lambda_{\max}(\cQ)}{2}\right)$
.
Then, $\{x^k\}$  converges to a solution to problem \eqref{3blkdual}. Moreover, the sequence  $\{u^k\}$ defined by \eqref{sequk} converges to a unique point, say $u^{\infty}$, such that $0\in\Gamma(u^{\infty})$.
\end{prop}

Since that Algorithm \ref{algo2} is intentionally designed for problem \eqref{cqpd}, Proposition \ref{thm:conv1} is still not enough for this algorithm.
Therefore, we need to further analyze its convergence properties.
The following theorem fulfils this objective.
\begin{thm}
Suppose that the solution set to problem \eqref{3blkdual} is nonempty and Assumption \ref{ass2} holds.
Let
$\sigma\in \left(0,\frac{2}{\lambda_{\max}(\cQ)}\right)$
and
$\rho\in\left(0,\frac{4-\sigma\lambda_{\max}(\cQ)}{2}\right)$.
Then, the infinite sequences $\{w^k\}$, $\{y^k\}$, $\{z^k\}$ and $\{x^k\}$ can be generated by Algorithm \ref{algo2}, and the sequence $\{(w^k,y^k,z^k)\}$ converges to a solution to problem \eqref{cqpd} while the sequence $\{x^k\}$ converges to a solution to problem \eqref{cqp}.
\end{thm}
\begin{proof}
According to Proposition \ref{thm:conv1} we know that both sequences $\{x^k\}$ and $\{u^k\}$ are convergent, especially that
$\{x^k\}$ converges to a solution of problem \eqref{cqp}.
Define  $x^{\infty}:=\lim_{k\to\infty}x^k$ and
$u^{\infty}:=\lim_{k\to\infty}u^k$.
Then, by \eqref{ykp2} we know that the sequence $\{y^k\}$ is convergent. Moreover, by \eqref{sequk} we know that $\{z^k\}$ is also convergent.
We define $y^{\infty}:=\lim_{k\to\infty}y^k$
and $z^{\infty}:=\lim_{k\to\infty}z^k$.
Note that
$$
\lim_{k\to\infty} (\rho\cA^*y^{k+1}-(1-\rho)z^k+z^{k+1}-\rho\cQ w^{k+1}-\rho c)
=0,
$$
which implies that  $\cA^*y^{\infty}+z^{\infty}-\cQ x^{\infty}-c=0$.
Then, by taking limits on both sides of \eqref{optt}, one has that $\cA x^{\infty}-b=0$.
Also, one can take limits in \eqref{optcond1} and obtains that
$0\in-\partial \theta(-z^{\infty})+ x^{\infty}$.
Therefore, by denoting $w^{\infty}=x^{\infty}$ we can conclude that $(w^\infty,y^{\infty},z^{\infty},x^{\infty})$ is a solution to the KKT system \eqref{KKT2}, so that
$\{(w^k,y^k,z^k)\}$ converges to a solution to problem \eqref{cqpd}.
 This completes the proof.
\qed
\end{proof}

\section{Conclusions}
\label{conclusion}
In this paper, we have shown that the modified $3$-block ADMM in Chang et al. \cite{modified admm} is an instance of the $3$-operator splitting scheme in \cite{Three-Operator}. Based on this observation, we considered a generalized modified $3$-block ADMM applied to the more general convex composite quadratic programming model, and derived its convergence via a very concise approach.
The obtained results paved the way for further study of the proposed generalized modified ADMM for convex composite quadratic programming such as the iteration complexity and the local or global convergence rate, which we leave as our future work.

\begin{acknowledgements}
The authors would like to thank Prof. Defeng Sun at The Hong Kong Polytechnic University for insightful comments and suggestions on this manuscript.
The research of X. Chang was supported by the Fundamental Research Funds for Central Universities and the Innovation Fund of Xidian University.
The research of L. Chen was supported by the Fundamental Research Funds for Central Universities.
\end{acknowledgements}

\small
\def\baselinestretch{1.3}

\end{document}